\documentclass[12pt, reqno]{amsart}
\usepackage{amssymb,amsthm,amsmath}
\usepackage[numbers,sort&compress]{natbib}
\usepackage{amssymb,amsmath}
\usepackage{amsfonts}
\usepackage{mathrsfs}
\usepackage{latexsym}
\usepackage{amssymb}
\usepackage{amsthm}
\usepackage{indentfirst}
\let\oldsection\section
\renewcommand\section{\setcounter{equation}{0}\oldsection}

\newtheorem{theorem}{Theorem}[section]
\newtheorem{lemma}{Lemma}[section]
\newtheorem{proposition}{Proposition}[section]
\newtheorem{definition}{Definition}[section]

\newtheorem{remark}{Remark}[section]

\date{April 11, 2014}

\begin{document}

\title[Global Strong Solutions to Incompressible Ericksen-Leslie System in $\mathbb R^3$]{Global Strong Solutions to Incompressible Ericksen-Leslie System in $\mathbb R^3$}

\author{Wenya~Ma}
\address[Wenya~Ma]{College of Information and Management of Science, Henan Agricultural University, Zhengzhou 450002, P.R. China}
\email{wyma@lsec.cc.ac.cn }

\author{Jinkai~Li}
\address[Jinkai~Li]{Department of Computer Science and Applied Mathematics, Weizmann Institute of Science, Rehovot 76100, Israel}
\email{jklimath@gmail.com}

\author{Huajun~Gong}
\address[Huajun Gong]{The Institute of Mathematical Sciences, University of Science and Technology of China, Anhui 230026, P.R. China }
\email{huajun84@hotmail.com }

\subjclass[2010]{AMS 35Q35, 76D03.}
\keywords{well-posedness; strong solution; liquid crystals.}

\allowdisplaybreaks
\begin{abstract}
In this paper, we consider the Cauchy problem to the Ericksen-Leslie system of liquid crystals in $\mathbb R^3$. Global well-posedness of strong solutions are obtained under the condition that the product of $\|u_0\|_2+\|\nabla d_0\|_2$ and $\|\nabla u_0\|_2+\|\nabla^2d_0\|_2$ is suitably small. This result can be viewed as a supplement to the local existence and blow up criteria discussed in \cite{Hong1}.
\end{abstract}

\maketitle

\allowdisplaybreaks

\section{Introduction}
\vskip 0.5cm \indent Continuum theory for nematic liquid crystals
was initiated by Oseen\cite{os} in static version, and
was reformulated by Frank\cite{fr}. Ericksen\cite{er1, er3} and Leslie\cite{le2, le3} proposed the corresponding dynamic model by extending their work, and the model can be successfully used to model the situation without defects.
The Oseen-Frank free energy of liquid crystal occupied in region $\Omega\subset \mathbb R^3$ with a conf\mbox{}iguration $d\in H^1(\Omega; S^2)$ is
$$E(d; \Omega)=\int_\Omega W(d, \nabla d)\mathrm d x,$$
where the Oseen-Frank density is given by
$$W(d, \nabla d)=k_1(\mathrm{div} d)^2+k_2(d\cdot\mathrm{curl} d)^2+k_3|d\times\mathrm{curl}d|^2$$
for positive constants $k_i, i=1,2,3$.
One can refer to \cite{Lin0, Hong0} for more details on static theory of liquid crystals.

In general, the incompressible Ericksen-Leslie model reads as follows,
\begin{subequations}
\begin{eqnarray} &&\partial_tu^i+(u\cdot\nabla)u^i-\Delta u^i+\partial_iP=-\partial_j(\partial_id^kW_{p_j^k}(d, \nabla d)),\\
&&\qquad\nabla \cdot u=0,\qquad |d|=1,\\
&&\partial_td^i+(u\cdot\nabla) d^i =\partial_j(W_{p^i_{j}}(d, \nabla d))-W_{d^i}(d, \nabla d)\nonumber\\
&&\qquad\qquad\qquad\quad\,\,-(\partial_j(W_{p^k_{j}}(d, \nabla d))-W_{d^k}(d, \nabla d))d^kd^i,
\end{eqnarray}\label{main}
\end{subequations}
where $u=(u^1, u^2, u^3)$ is the velocity of the f\mbox{}luid, $d=(d^1, d^2, d^3)$ is the unit molecular direction, and $P$ is the pressure. Additionally,
$i, j, k=1, 2, 3$, and the Einstein summation is used.

The above Ericksen-Leslie model is a coupled system by Navier-Stokes equations and the gradient f\mbox{}low of the Oseen-Frank model. Both of their developments are heuristic to our further discussion. The well-known result on Navier-Stokes equations is about the existence and partial regularity of global suitable weak solutions presented in \cite{ckn, Lin1}, and the Serrin or Beale-Kato-Majda type blow up criteria presented in \cite{Serrin, bkm}. The development related to the heat f\mbox{}low of harmonic maps, which is a specif\mbox{}ic example of the above gradient f\mbox{}low, is the existence and partial regularity of global weak solutions presented in \cite{Struwe, CStruwe}. Accordingly, many literatures studied the simplif\mbox{}ied version of nematic liquid crystal equations, where required $k_1=k_2=k_3$. In particular, the local strong solutions were obtained in \cite{Wd}, and the blow up criteria in \cite{huang}. For the more general case, \cite{Hong2} started the study on existence and regularity in $\mathbb R^2$, and \cite{Hong1} on local existence and blow up criteria in $\mathbb R^3$.

As discussed in \cite{Hong1}, we consider the Cauchy problem to the general Ericksen-Leslie system above in $\mathbb R^3$.
The following initial data are imposed to (\ref{main}):
\begin{equation}\label{initial cond}
(u, d)|_{t=0}=(u_0, d_0).
\end{equation}
Moreover, we always suppose without any further mention that the initial data $u_0$ and $d_0$ satisfy
\begin{equation}
  \label{assum}
  u_0\in H^1_\sigma(\mathbb R^3), \quad d_0-d_0^*\in H^2(\mathbb R^3), \quad |d_0|=1,
\end{equation}
where $d_0^*$ is a constant unit vector, and $H^1_\sigma(\mathbb R^3)=\{v\in H^1(\mathbb R^3)|\,\, \mathrm{div} v=0\}$.

Throughout this paper, we use $C$ for a generic positive constant which may change from line to line, $\|\cdot\|_{q}$ for the $L^q(\mathbb R^3)$ norm with $q\geq1$, and $\int\!|\cdot|\mathrm dx$ for $\int_{\mathbb R^3}\!\!|\cdot|\mathrm dx$.

This paper is devoted to global existence of strong solutions under certain smallness conditions.
For convenience, we present def\mbox{}initions related to global strong solutions before statement of our main result.

\begin{definition}
Given $T>0,$ a couple $(u, d)$ is a strong solution to system (\ref{main})-(\ref{initial cond}) on $\mathbb R^3\times(0, T),$ if it has the regularity properties
$$u, \nabla d\in C([0, T]; H^1(\mathbb R^3))\cap L^2(0, T; H^2(\mathbb R^3)),$$
$$\partial_t u\in L^2(0, T; L^2(\mathbb R^3)),\, \partial_t d\in L^2(0, T; H^1(\mathbb R^3);$$
and it satisf\mbox{}ies (\ref{main}) a.e. on $\mathbb R^3\times(0, T)$, and the initial condition (\ref{initial cond}).
\end{definition}
\begin{definition}
A f\mbox{}inite positive number $\mathcal T$ is called the maximal existence time of a strong solution $(u, d)$ to system (\ref{main})-(\ref{initial cond}) on $\mathbb R^3\times(0, T)$, if for any $T<\mathcal T$, $(u, d)$ is a strong solution to system (\ref{main})-(\ref{initial cond}) on $\mathbb R^3\times(0, T)$, and
$$\lim\limits_{T\to\mathcal T\!\!-}\sup\limits_{0\leq t\leq T}(\|u(t)\|_{H^1}^2+\|\nabla d(t)\|_{H^1}^2)=\infty.$$
\end{definition}
\begin{definition}
A couple $(u, d)$ is called a global strong solution to system (\ref{main})-(\ref{initial cond}) on $\mathbb R^3\times(0, \infty),$ if it is a strong solution to system (\ref{main})-(\ref{initial cond}) on $\mathbb R^3\times(0, T)$ for any f\mbox{}inite time $T$.
\end{definition}

\begin{theorem}\label{thm}
Under the condition (\ref{assum}), system (\ref{main})-(\ref{initial cond}) has a unique global strong solution, provided
$$(\|u_0\|_2+\|\nabla d_0\|_2)(\|\nabla u_0\|_2+\|\nabla^2 d_0\|_2)\leq\varepsilon_0,$$
where $\varepsilon_0$ is a small positive constant depending only on $k_1, k_2, k_3$.
\end{theorem}

\begin{remark}
$(\|u\|_2+\|\nabla d\|_2)(\|\nabla u\|_2+\|\nabla^2 d\|_2)$ is a scaling invariance under the transform
$$u_\lambda(x, t)=\lambda u(\lambda x, \lambda^2t), \quad d_\lambda(x, t)= d(\lambda x, \lambda^2t).$$
Therefore, our result can be viewed as the global existence of strong solutions in critical space.
\end{remark}

\section{Proof of Theorem \ref{thm}}
\setcounter{equation}{0}
Recalling the expression of $W(d,\nabla d)$ in the introduction, one can easily check that
\begin{eqnarray*}
&&W(z, p)\geq a|p|^2,\quad\,\, W_{p_\alpha^ip_\beta^j}(z, p)\xi_\alpha^i\xi_\beta^j\geq a|\xi|^2,
\end{eqnarray*}
for any $z\in\mathbb R^3, p, \xi\in\mathbb M^{3\times3},$ where $a=\min\{k_1,k_2,k_3\}$, and
\begin{eqnarray*}
&&|W(d, \nabla d)|\leq C|d|^2|\nabla d|^2,\qquad |W_{d^i}(d, \nabla d)|\leq C|d||\nabla d|^2,\\
&&|W_{d^id^j}(d, \nabla d)|\leq C|\nabla d|^2,\qquad\, |W_{p_\alpha^i}(d, \nabla d)|\leq C|d|^2|\nabla d|,\\
&&|W_{p_\alpha^ip_\beta^j}(d, \nabla d)|\leq C|d|^2,\qquad\,\, |W_{d^ip_\beta^j}(d, \nabla d)|\leq C|d||\nabla d|,
\end{eqnarray*}
for a positive constant $C$ depending only on $k_1, k_2, k_3$. These inequalities will be used frequently without any further mention in this paper.

We f\mbox{}irst cite the following local existence and blow up criteria of strong solutions, which is a special case of those in Hong-Li-Xin \cite{Hong1}.

\begin{lemma} (Local existence and blow up criteria of strong solutions)
\label{lem0}
Suppose that the condition (\ref{assum}) holds true. Then system (\ref{main})--(\ref{initial cond}) has a unique local strong solution $(u,d)$ on $\mathbb R^3\times(0,T)$, for a positive number $T$ depending only on the initial data and $k_1,k_2,k_3$.
\end{lemma}

For strong solutions, we have the following basic energy balance law.

\begin{lemma}(Basic energy balance law, see e.g., Hong \cite{Hong2})\label{lem1} Let $(u,d)$ be a strong solution to system (\ref{main})--(\ref{initial cond}) on $\mathbb R^3\times(0,T)$. Then
\begin{equation}\frac{\mathrm d}{\mathrm dt}\int\left(\frac{|u|^2}{2}+W(d, \nabla d)\right)\mathrm dx+\int(|\nabla u|^2+|\partial_td+(u\cdot\nabla) d|^2)\mathrm dx=0,
\end{equation}
for any $t\in(0,T)$.
\end{lemma}

 However, the estimate on the second order spatial derivatives of the director f\mbox{}ield $d$ was not included in the basic energy balance law. This estimate is given by the following lemma.

\begin{lemma}(First order energy inequality)\label{lem2}
Let $(u,d)$ be a strong solution to system (\ref{main})--(\ref{initial cond}) on $\mathbb R^3\times(0,T)$. Then
\begin{equation}
\frac{\mathrm d}{\mathrm dt}\int|\nabla d|^2\mathrm dx+2a\int|\nabla^2 d|^2\mathrm dx\leq C\int(|u|^2+|\nabla d|^2)|\nabla^2 d|\mathrm dx,
\end{equation}
for any $t\in(0,T)$.
\end{lemma}

\begin{proof}
Multiplying (\ref{main}c) by $\Delta d^i$, and integrating on $\mathbb R^3$, then we get
\begin{eqnarray}
&&\frac12\frac{\mathrm d}{\mathrm dt}\int|\nabla d|^2\mathrm dx+\int\partial_\alpha W_{p_{_\alpha}}(d, \nabla d)\cdot\Delta d\mathrm dx\nonumber\\
&=&\int\left\{(u\cdot\nabla) d+W_d(d, \nabla d)+\left[(\partial_\alpha W_{p_{_\alpha}}(d, \nabla d)-W_d(d, \nabla d))\cdot d\right]d\right\}\cdot\Delta d\mathrm dx\nonumber\\
&=&\int\left[((u\cdot\nabla) d+W_d(d, \nabla d))\cdot\Delta d-(\partial_\alpha W_{p_{_\alpha}}(d, \nabla d)-W_d(d, \nabla d))\cdot d|\nabla d|^2\right]\mathrm dx\nonumber\\
&\leq&C\int\left[(|u||\nabla d|+|\nabla d|^2)|\Delta d|+(|\nabla^2 d|+|\nabla d|^2)|\nabla d|^2\right]\mathrm dx\nonumber\\
&\leq&C\int\left[(|u|^2+|\nabla d|^2)|\nabla^2 d|+|\nabla d|^4\right]\mathrm dx\nonumber\\
&\leq&C\int(|u|^2+|\nabla d|^2)|\nabla^2 d|\mathrm dx,\label{f1}
\end{eqnarray}
where in the last step we have used the fact $|\nabla d|^2=-d\cdot\Delta d.$

It follows from integrating by parts that
\begin{eqnarray}
\int\partial_\alpha W_{p_{_\alpha}}(d, \nabla d)\cdot\Delta d\mathrm dx&=&\int\partial_\beta W_{p_{_\alpha}}(d, \nabla d)\cdot\partial^2_{\alpha\beta} d\mathrm dx\nonumber\\
&=&\int [W_{p_{_\alpha}p_{_\gamma}}(d, \nabla d)\partial^2_{\gamma\beta}d+W_{p_{_\alpha}d^j}(d, \nabla d)\partial_\beta d^j]\cdot\partial^2_{\alpha\beta} d\mathrm dx\nonumber\\
&\geq&a\int|\nabla^2 d|^2\mathrm dx-C\int|\nabla d|^2|\nabla^2 d|\mathrm dx.\label{f2}
\end{eqnarray}
So combining (\ref{f1}) with (\ref{f2}), we get
\begin{equation*}
\frac{\mathrm d}{\mathrm dt}\int|\nabla d|^2\mathrm dx+2a\int|\nabla^2 d|^2\mathrm dx\leq C\int(|u|^2+|\nabla d|^2)|\nabla^2 d|\mathrm dx,
\end{equation*}
proving the conclusion.
\end{proof}

\begin{lemma}(Second order energy inequality, see Lemma 3.2 in Hong-Li-Xin \cite{Hong1})\label{lem3}
Let $(u,d)$ be a strong solution to system (\ref{main})--(\ref{initial cond}) on $\mathbb R^3\times(0,T)$. Then
\begin{eqnarray*}
&&\frac{\mathrm d}{\mathrm dt}\int(|\nabla u|^2+|\nabla^2 d|^2)\mathrm dx+\int\left(|\nabla^2 u|^2+\frac32a|\nabla^3 d|^2\right)\mathrm dx\\
&\leq& C\int(|u|^2+|\nabla d|^2)(|\nabla u|^2+|\nabla^2 d|^2)\mathrm dx,
\end{eqnarray*}
for any $t\in(0,T)$.
\end{lemma}

The proof of Theorem 1.1 relies on the following two propositions.
\begin{proposition}\label{prop1}
Let $(u,d)$ be a strong solution to system (\ref{main})--(\ref{initial cond}) on $\mathbb R^3\times(0,T)$. For any $t\in(0,T)$, def\mbox{}ine
$$
m(t)=\sup\limits_{0\leq s\leq t}(\|u(s)\|_2+\|\nabla d(s)\|_2)(\|\nabla u(s)\|_2+\|\nabla^2 d(s)\|_2).
$$
Then we have
\begin{eqnarray*}
&&\frac{\mathrm d}{\mathrm dt}\int(|u|^2+|\nabla d|^2+2W(d, \nabla d))\mathrm dx+\int(|\nabla u|^2+a|\nabla^2 d|^2)\mathrm dx\leq 0,\\
&&\frac{\mathrm d}{\mathrm dt}\int(|\nabla u|^2+|\nabla^2 d|^2)\mathrm dx+\frac12\int(|\nabla^2 u|^2+a|\nabla^3 d|^2)\mathrm dx\leq0,
\end{eqnarray*}
for any $t\in(0,T)$, as long as $ m(t)\leq\varepsilon_1,$ where $\varepsilon_1$ is a small constant depending only on $a$.
\end{proposition}

\begin{proof}
By virtue of Lemma \ref{lem1} and \ref{lem2}, we have
\begin{eqnarray}
&&\frac{\mathrm d}{\mathrm dt}\int(|u|^2+|\nabla d|^2+2W(d, \nabla d))\mathrm dx+\int(2|\nabla u|^2+2a|\nabla^2 d|^2)\mathrm dx\nonumber\\
&\leq& C\int(|u|^2+|\nabla d|^2)|\nabla^2 d|\mathrm dx.\label{p1}
\end{eqnarray}
The right-hand term is then estimated by H\"older and Sobolev embedding inequalities
\begin{eqnarray*}
\int(|u|^2+|\nabla d|^2)|\nabla^2 d|\mathrm dx&\leq& C(\|u\|_4^2+\|\nabla d\|_4^2)\|\nabla^2 d\|_2\\
&\leq& C(\|u\|_2+\|\nabla d\|_2)^\frac12(\|u\|_6+\|\nabla d\|_6)^\frac32\|\nabla^2 d\|_2\\
&\leq& C(\|u\|_2+\|\nabla d\|_2)^\frac12(\|\nabla u\|_2+\|\nabla^2 d\|_2)^\frac32\|\nabla^2 d\|_2\\
&\leq& C(\|u\|_2+\|\nabla d\|_2)^\frac12(\|\nabla u\|_2+\|\nabla^2 d\|_2)^\frac52\\
&\leq& Cm(t)^\frac12(\|\nabla u\|^2_2+\|\nabla^2 d\|^2_2)\\
&\leq& C\varepsilon_1^\frac12(\|\nabla u\|^2_2+\|\nabla^2 d\|^2_2)\\
&\leq&\|\nabla u\|^2_2+a\|\nabla^2 d\|^2_2,
\end{eqnarray*}
as long as $m(t)\leq\varepsilon_1$, and $\varepsilon_1$ is small enough. Thus it follows from (\ref{p1}) that
\begin{eqnarray*}
\frac{\mathrm d}{\mathrm dt}\int(|u|^2+|\nabla d|^2+2W(d, \nabla d))\mathrm dx+\int(|\nabla u|^2+a|\nabla^2 d|^2)\mathrm dx\leq0,
\end{eqnarray*}
the f\mbox{}irst conclusion holds.

By virtue of Lemma \ref{lem3}, we have
\begin{eqnarray}
&&\frac{\mathrm d}{\mathrm dt}\int(|\nabla u|^2+|\nabla^2 d|^2)\mathrm dx+\int(|\nabla^2 u|^2+\frac32a|\nabla^3 d|^2)\mathrm dx\nonumber\\
&\leq& C\int(|u|^2+|\nabla d|^2)(|\nabla u|^2+|\nabla^2 d|^2)\mathrm dx.\label{plem3}
\end{eqnarray}
By the H\"older and Sobolev embedding inequalities, we have
\begin{eqnarray}
&&\int(|u|^2+|\nabla d|^2)(|\nabla u|^2+|\nabla^2 d|^2)\mathrm dx\nonumber\\
&\leq&C(\|u\|_6^2+\|\nabla d\|_6^2)(\|\nabla u\|_2+\|\nabla^2 d\|_2)(\|\nabla u\|_6+\|\nabla^2 d\|_6)\nonumber\\
&\leq&C(\|\nabla u\|^2_2+\|\nabla^2 d\|^2_2)(\|\nabla u\|_2+\|\nabla^2 d\|_2)(\|\nabla^2 u\|_2+\|\nabla^3 d\|_2).\label{p2}
\end{eqnarray}
It follows from integrating by parts that
\begin{eqnarray}
&&\int(|\nabla u|^2+|\nabla^2 d|^2)\mathrm dx=\int(\partial_iu\cdot\partial_i u+\partial_{ij}^2d\cdot\partial_{ij}^2d)\mathrm dx\nonumber\\
&=&-\int (u\cdot\Delta u+\partial_i d\cdot\partial_i \Delta d)\mathrm dx\leq (\|u\|_2+\|\nabla d\|_2)(\|\nabla^2 u\|_2+\|\nabla^3 d\|_2),
\end{eqnarray}
and thus, recalling (\ref{p2}), we have
\begin{eqnarray}
&&\int(|u|^2+|\nabla d|^2)(|\nabla u|^2+|\nabla^2 d|^2)\mathrm dx\nonumber\\
&\leq&C(\|u\|_2+\|\nabla d\|_2)(\|\nabla u\|_2+\|\nabla^2 d\|_2)(\|\nabla^2 u\|^2_2+\|\nabla^3 d\|^2_2)\nonumber\\
&\leq&Cm(t)(\|\nabla^2 u\|^2_2+\|\nabla^3 d\|^2_2)\leq C\varepsilon_1(\|\nabla^2 u\|^2_2+\|\nabla^3 d\|^2_2)\nonumber\\
&\leq&\frac12(\|\nabla^2 u\|^2_2+2a\|\nabla^3 d\|^2_2),
\end{eqnarray}
provided $m(t)\leq\varepsilon_1$, and $\varepsilon_1$ is small.
Substituting this inequality into (\ref{plem3}), one gets the second conclusion.
\end{proof}

By the aid of the above proposition, we can establish the uniform estimates in time on the strong solutions of (\ref{main})-(\ref{initial cond}) as in the following proposition.

\begin{proposition}\label{prop2}
Let $(u, d)$ be a strong solution to system (\ref{main})-(\ref{initial cond}) on $\mathbb R^3\times(0, T)$. Then there is a small positive constant $\varepsilon_0$, such that if
$$(\|u_0\|_2+\|\nabla d_0\|_2)(\|\nabla u_0\|_2+\|\nabla^2 d_0\|_2)\leq \varepsilon_0,$$
then
$$\sup\limits_{0\leq s\leq t}(\|u\|_{H^1}^2+\|\nabla d\|_{H^1}^2)+\int_0^t(\|\nabla u\|_{H^1}^2+\|\nabla^2 d\|_{H^1}^2)\mathrm ds\leq C(\|u_0\|_{H^1}^2+\|\nabla d_0\|_{H^1}^2),$$
for any $t\in(0,T)$, where $C$ is a positive constant independent of $T$.
\end{proposition}

\begin{proof}
Let $m(t)$ be the function as def\mbox{}ined in Proposition \ref{prop1}. Recalling the regularity properties of strong solutions and the def\mbox{}inition of $m(t)$, one can easily see that $m(t)$ is continuously increasing on $[0, T)$. By assumption, it is obvious that $m(0)\leq \varepsilon_0.$
Def\mbox{}ine $T_0=\sup\{t\in (0, T)\,|\,m(t)\leq \varepsilon_1\}$.
By Proposition \ref{prop1}, we have
\begin{eqnarray*}
&&\frac{\mathrm d}{\mathrm dt}\int(|u|^2+|\nabla d|^2+2W(d, \nabla d))\mathrm dx+\int(|\nabla u|^2+a|\nabla^2 d|^2)\mathrm dx\leq 0,\\
&&\frac{\mathrm d}{\mathrm dt}\int(|\nabla u|^2+|\nabla^2 d|^2)\mathrm dx+\frac12\int(|\nabla^2 u|^2+a|\nabla^3 d|^2)\mathrm dx\leq0,
\end{eqnarray*}
for any $t\in (0, T_0)$.

Integrating the above two inequalities in $t$, we have
\begin{equation}
\sup\limits_{0\leq s\leq t}(\|u\|^2_2+\|\nabla d\|^2_2)+\int_0^t(\|\nabla u\|^2_2+\|\nabla^2 d\|^2_2)\mathrm ds\leq C(\|u_0\|^2_2+\|\nabla d_0\|^2_2),\label{p2f1}
\end{equation}
and
\begin{equation}
\sup\limits_{0\leq s\leq t}(\|\nabla u\|^2_2+\|\nabla^2 d\|^2_2)+\int_0^t(\|\nabla^2 u\|^2_2+\|\nabla^3 d\|^2_2)\mathrm ds
\leq C(\|\nabla u_0\|^2_2+\|\nabla^2 d_0\|^2_2),\label{p2f2}
\end{equation}
for any $t\in (0, T_0)$.

Then it follows from the above two estimates that, for any $t\in (0, T_0)$,
\begin{eqnarray*}
m(t)&\leq&\sup\limits_{0\leq s\leq t}(\|u(s)\|_2+\|\nabla d(s)\|_2)\sup\limits_{0\leq s\leq t}(\|\nabla u(s)\|_2+\|\nabla^2 d(s)\|_2)\\
&\leq&C(\|u_0\|_2+\|\nabla d_0\|_2)(\|\nabla u_0\|_2+\|\nabla^2 d_0\|_2)\leq C\varepsilon_0\leq\frac{\varepsilon_1}2,
\end{eqnarray*}
by choosing $\varepsilon_0\leq\frac{\varepsilon_1}{2C}$, where $\varepsilon_1$ is the constant in Proposition \ref{prop1}.

If $T_0<T,$ then the above estimates implies that there is a $T_1>T_0$, such that
$$\sup\limits_{0\leq s\leq T_1}m(t)\leq \varepsilon_1,$$
from which, recalling the def\mbox{}inition of $T_0,$ it has $T_1\leq T_0,$ contradicting to $T_1>T_0.$
Thus it must have $T_0=T,$ and consequently (\ref{p2f1}) and (\ref{p2f2}) hold for any $t\in(0, T).$
Therefore,
$$\sup\limits_{0\leq s\leq t}(\|u\|_{H^1}^2+\|\nabla d\|_{H^1}^2)+\int_0^t(\|\nabla u\|_{H^1}^2+\|\nabla^2 d\|_{H^1}^2)\mathrm ds\leq C(\|u_0\|_{H^1}^2+\|\nabla d_0\|_{H^1}^2),$$
proving the conclusion.
\end{proof}

\begin{proof}[\textbf{Proof of Theorem \ref{thm}}]
Let $\varepsilon_0$ be the constant in Proposition \ref{prop2}, and suppose the conditions in Theorem \ref{thm} hold. By the local existence result, Lemma \ref{lem0}, there is a local strong solution $(u,d)$ to system (\ref{main})--(\ref{initial cond}). Extend such local solution to the maximal existence time $T^*$. We are going to prove that $T^*=\infty$. Suppose, by contradiction, that $T^*<\infty$.
By Proposition \ref{prop2}, it has
$$\sup\limits_{0\leq s\leq t}(\|u\|_{H^1}^2+\|\nabla d\|_{H^1}^2)+\int_0^t(\|\nabla u\|_{H^1}^2+\|\nabla^2 d\|_{H^1}^2)\mathrm ds\leq C(\|u_0\|_{H^1}^2+\|\nabla d_0\|_{H^1}^2),$$
for any $t\in(0,T^*)$. On account of this estimates uniform in time, one can apply the local existence result, Lemma \ref{lem0}, to extend $(u,d)$ to a strong solution beyond $T^*$, contradicting to the def\mbox{}inition of $T^*$. This contradiction implies that $T^*=\infty$, and thus $(u,d)$ is a global strong solution. The uniqueness of global strong solutions is a direct consequence of that for local ones. This completes the proof.
\end{proof}

\section*{Acknowledgments}
{This work is partially supported by the Institute of Mathematical Sciences, the Chinese University of Hong Kong.}
\par

\end{document}